\documentclass{amsart}[12pt]
\usepackage{amsmath,amssymb,amsthm,hyperref,mathtools}
\usepackage{color}
\usepackage{cleveref}

\usepackage{tikz}
\newtheorem{thm}{Theorem}[section]
\newtheorem{cor}[thm]{Corollary}

\newtheorem{lem}[thm]{Lemma}
\newtheorem{prop}[thm]{Proposition}
\theoremstyle{remark}

\theoremstyle{definition}

\newcommand{\Cay}[2]{\mathrm{Cay}(#1,#2)}		
\newcommand{\Aut}[1]{\mathrm{Aut}(#1)}			

\title{$\mathbb Z_3^8$ is not a CI-group}

\author{Joy Morris}

\thanks{Supported by the Natural Science and Engineering Research Council of Canada (grant RGPIN-2017-04905).}

\address{Department of Mathematics and Computer Science\\
	University of Lethbridge\\
	Lethbridge, AB T1K 3M4\\
	Canada}

\email{joy.morris@uleth.ca}

\begin{document}

\begin{abstract}
A Cayley graph $\Cay{G}{S}$ has the CI (Cayley Isomorphism) property if for every isomorphic graph $\Cay{G}{T}$, there is a group automorphism $\alpha$ of $G$ such that $S^\alpha=T$. The DCI (Directed Cayley Isomorphism) property is defined analogously on digraphs. A group $G$ is a CI-group if every Cayley graph on $G$ has the CI property, and is a DCI-group if every Cayley digraph on $G$ has the DCI property. Since a graph is a special type of digraph, this means that every DCI-group is a CI-group, and if a group is not a CI-group then it is not a DCI-group, but there are well-known examples of groups that are CI-groups but not DCI-groups.

In 2009, Spiga showed that $\mathbb Z_3^8$ is not a DCI-group, by producing a digraph that does not have the DCI property. He also showed that $\mathbb Z_3^5$ is a DCI-group (and therefore also a CI-group). Until recently the question of whether there are elementary abelian $3$-groups that are not CI-groups remained open. In a recent preprint with Dave Witte Morris, we showed that $\mathbb Z_3^{10}$ is not a CI-group. In this paper we show that with slight modifications, the underlying undirected graph of order $3^8$ described by Spiga is does not have the CI property, so $\mathbb Z_3^8$ is not a CI-group.
\end{abstract}



\maketitle

\section{Introduction}
Let $G$ be a group, and $S \subseteq G$. The Cayley digraph $\Cay{G}{S}$ is the digraph whose vertices are the elements of the group $G$, with an arc from $g$ to $h$ if and only if $hg^{-1} \in S$. If $S=S^{-1}$ then for any arc from $g$ to $h$ there is a paired arc from $h$ to $g$; we replace these two arcs with a single undirected edge and call the resulting structure the Cayley graph $\Cay{G}{S}$. It is well-known that $\Gamma$ is isomorphic to a Cayley graph on $G$ if and only if $\Aut{\Gamma}$ has a regular subgroup isomorphic to $G$. Typically we abuse notation by ignoring the isomorphism and simply saying that $\Gamma$ is a Cayley graph on $G$ in this event.

The Cayley graph $\Cay{G}{S}$ has the CI (Cayley Isomorphism) property if for every isomorphic graph $\Cay{G}{T}$, there is a group automorphism $\alpha$ of $G$ such that $S^\alpha=T$. Note that whenever $\alpha$ is an automorphism of $G$, it induces a graph isomorphism from $\Cay{G}{S}$ to $\Cay{G}{S^\alpha}$. This means that essentially, a Cayley graph has the CI property if all of its isomorphisms to other Cayley graphs on the same group have algebraic justifications: group automorphisms that induce not necessarily that particular isomorphism, but an isomorphism to the same graph. The CI problem is the problem of determining which graphs have the CI property. The DCI (Directed Cayley Isomorphism) property and the DCI problem are defined analogously on digraphs. 

Although work on the (D)CI problem dates back at least to 1967 \cite{Adam}, the standard terminology was coined and fundamental results about the problem were proved by Babai in \cite{babai}. A group $G$ is a CI-group if every Cayley graph on $G$ has the CI property, and is a DCI-group if every Cayley digraph on $G$ has the DCI property. Since a graph is a special type of digraph, this means that every DCI-group is a CI-group, and if a group is not a CI-group then it is not a DCI-group, but there are well-known examples of groups that are CI-groups but not DCI-groups. For example, Muzychuk \cite{Muzy-circ-1,Muzy-circ-2} characterised cyclic groups according to which are DCI and which are CI: the cyclic group of order $n$ is a DCI-group if and only if $n\in \{k,2k,4k\}$ where $k$ is odd and square-free. It is a CI-group if and only if it is a DCI-group, or $n \in \{8,9,18\}$.

Once cyclic groups were completely understood with respect to the CI and DCI problems, elementary abelian groups became a natural class of groups to consider. This class of groups has become even more fundamentally important in understanding the (D)CI problem over time, since the combined work of a number of researchers has shown that any (D)CI group is a direct product of up to three factors, each of which is either small, abelian, or the semidirect product of an abelian group with a small cyclic group. (See for example \cite{Ted,TedMishaPablo,LiLuPalfy}, although Math Reviewers have noted some errors in the statements of the relevant results.)

In a 2003 paper, Muzychuk \cite{Muzy-elem-abel} proved that an elementary abelian group of sufficiently high rank is not a DCI-group. He did not consider the undirected problem in that paper, and the rank he achieved for an elementary abelian $p$-group was $2p-1+\binom{2p-1}{p}$. Spiga \cite{Spiga-CI} improved this rank to $4p-2$. Both of these papers may have introduced some confusion into the problem as they talk about the CI problem and property but use directed graphs throughout, so in fact prove that an elementary abelian $p$-group of sufficiently high rank is not a DCI-group although the statement of their results say that this is not a CI-group. Somlai \cite{Somlai} addressed this issue, and improved the previous results by showing that an elementary abelian $p$-group of rank at least $2p+3$ is not a CI-group when $p \ge 5$. He did also prove a similar result when $p=3$ but in this case was only able to show that the group is not a DCI-group, and noted that the problem of the existence of a non-CI elementary abelian $3$-group remained open.
For the $p=3$ case, Spiga \cite{Spiga-CI-3} had also previously shown that $\mathbb Z_3^8$ is not a DCI-group, by producing a digraph that does not have the DCI property. 

On the other side of things, the best known result \cite{Istvan} shows that $\mathbb Z_p^5$ is a DCI-group. It is known \cite{Nowitz} that $\mathbb Z_2^6$ is not a CI-group, but there remains a gap in our knowledge for every prime $p>2$.

In a recent preprint with Dave Witte Morris \cite{Dave}, we show that $\mathbb Z_3^{10}$ is not a CI-group, and more generally demonstrate a method for using non-DCI digraphs to construct non-CI graphs whose order is a fairly small multiple of the order of the original digraph, and thereby find groups that are not CI-groups.

 In this paper we show that a slightly modified version of the underlying undirected graph of the non-CI digraph of order $3^8$ described by Spiga in \cite{Spiga-CI-3} does not have the CI property, so $\mathbb Z_3^8$ is not a CI-group.

\section{Two isomorphic Cayley graphs}

Let $\{w_1, w_2, w_3, v_1, v_2, v_3, v_4, v_5\}$ be a generating set for $G \cong \mathbb Z_3^8$. We define the following sets:
\begin{eqnarray*}
S_{0,0,0} &=& \{v_1-v_5, v_2+v_3-v_4+v_5, v_3-v_4+v_5, v_4+v_5, v_5\}\\
S_{1,0,0} &=& \{w_1+av_1+bv_2+cv_5 : a,b,c \in \mathbb Z_3\}\\
S_{0,1,0} &=& \{w_2+av_1+bv_3+cv_4+dv_5 : a,b,c,d \in \mathbb Z_3\}\\
S_{0,0,1} &=& \{w_3+av_2+bv_3+cv_4+dv_5 : a,b,c,d \in \mathbb Z_3\}\\
S_{1,1,0} &=& \{w_1+w_2+av_1+bv_2+cv_3+bv_4+dv_5 : a,b,c,d \in \mathbb Z_3\}\\
S_{1,0,1} &=& \{w_1+w_3+av_1+bv_2+av_3+cv_4+dv_5 : a,b,c,d \in \mathbb Z_3\}\\
S_{0,1,1} &=& \{w_2+w_3+av_1+bv_2+cv_3+dv_4-(a+b)v_5 : a,b,c,d \in \mathbb Z_3\}\\
S_{1,1,1} &=& \{w_1+w_2+w_3+av_1+bv_2+cv_3+dv_4+(-a-b+c+d)v_5 : a,b,c,d \in \mathbb Z_3\}\\
S_{2,1,1} &=& \{2w_1+w_2+w_3+av_1+bv_2+cv_3+dv_4-(a+b+c+d)v_5 : a,b,c,d \in \mathbb Z_3\}\\
S_{1,2,1} &=& \{w_1+2w_2+w_3+av_1+bv_2+cv_3+dv_4+(a+b-c+d)v_5 : a,b,c,d \in \mathbb Z_3\}\\
S_{1,1,2} &=& \{w_1+w_2+2w_3+av_1+bv_2+cv_3+dv_4+(a+b+c-d)v_5 : a,b,c,d \in \mathbb Z_3\}\\
S&=& S_{2,1,1}\cup S_{1,2,1} \cup S_{1,1,2}\bigcup_{0 \le i,j,k \le 1} S_{i,j,k}.
\end{eqnarray*}
For $0 \le i,j \le 1$ let $T_{i,j,k}=S_{i,j,k}$ except let $T_{1,1,1}=S_{1,1,1}+v_5$. Also let $$T_{2,1,1}=S_{2,1,1}-v_5, T_{1,2,1}=S_{1,2,1}-v_5, T_{1,1,2}=S_{1,1,2}-v_5.$$
Similar to $S$, let $$T= T_{2,1,1}\cup T_{1,2,1} \cup T_{1,1,2}\bigcup_{0 \le i,j,k \le 1} T_{i,j,k}.$$

The graphs we will be studying throughout this paper are $\Gamma_1=\Cay{G}{S \cup S^{-1}}$ and $\Gamma_2=\Cay{G}{T \cup T^{-1}}$. We will often abuse notation and terminology by conflating a group element with the corresponding vertex in one or both of these Cayley graphs.

For convenience of notation, we also define a partition of the elements of $G$ into subsets of cardinality $3^5$ by $$B_{i,j,k}=\{iw_1+jw_2+kw_3+av_1+bv_2+cv_3+dv_4+fv_5 : a,b,c,d,f \in \mathbb Z_3\}.$$

The digraphs studied in \cite{Spiga-CI-3} that are determined to be isomorphic but not via a group automorphism, and therefore not to have the DCI-property, are the Cayley digraphs 
$$\overrightarrow{\Gamma_1'}=\Cay{G}{(S\setminus S_{0,0,0}) \cup \{v_1,v_3,v_4,v_5\}}$$ 
$$\text{and }\overrightarrow{\Gamma_2'}=\Cay{G}{(T \setminus T_{0,0,0})\cup \{v_1,v_3,v_4,v_5\}}.$$

\begin{prop}\label{prop-iso}
There is a polynomial that maps $S\cup S^{-1}$ to $T\cup T^{-1}$ and acts as a graph isomorphism from $\Gamma_1$ to $\Gamma_2$. 
\end{prop}

\begin{proof}
According to \cite{Spiga-CI-3}, the map $\psi: G \to G$ defined by 
$$\psi(\textbf{v})=\textbf{v}+x_1x_2^2v_1+x_1x_3^2v_2+x_2^2x_3+x_2x_3^2v_4+x_1x_2x_3v_5$$
for each $$\textbf{v}=\sum_{i=1}^3 x_iw_i + \sum_{j=1}^5 y_jv_j \in G$$
is a digraph isomorphism from $\overrightarrow{\Gamma_1'}$ to $\overrightarrow{\Gamma_2'}$. 

Our graphs $\Gamma_1$ and $\Gamma_2$ are very close to being the underlying undirected graphs of these digraphs $\overrightarrow{\Gamma_1'}$ and $\overrightarrow{\Gamma_2'}$, which must be isomorphic via $\psi$ since the digraphs are. Indeed, aside from the edges and arcs that lie inside each $B_{i,j,k}$, they are the same. Since $\psi$ fixes every $B_{i,j,k}$, it must act as an isomophism from $\Gamma_1$ to $\Gamma_2$ with respect to every edge that is not contained within some $B_{i,j,k}$. It remains to be shown that the edges that come from $S_{0,0,0}^{\pm1}$ are preserved by $\psi$.

For any fixed $(i,j,k)$, $\psi$ adds $\sum_{m=1}^5c_mv_m$ to each vertex of $B_{i,j,k}$, for some constants $c_1$ through $c_5$ that depend only on $i,j,k$. Thus the action of $\psi$ on $B_{i,j,k}$ is a translation by some element of $\langle v_1, \ldots, v_5\rangle$. This means that for any choice of $S\cap B_{0,0,0}$, as long as $T\cap B_{0,0,0}=S\cap B_{0,0,0}$, $\psi$ must preserve the edges that lie within $B_{i,j,k}$ (the edges that come from $S_{0,0,0}^{\pm1}$). Thus, $\psi$ is indeed an isomorphism from $\Gamma_1$ to $\Gamma_2$.
\end{proof}

\section{They are not isomorphic via a group automorphism}

We begin with a number of assertions about the structure of mutual neighbours of vertices in these graphs. These serve to limit how an arbitrary isomorphism between the graphs can act, even more so in the case where the isomorphism must also be a group automorphism.

The bound provided in our first lemma will shortly allow us to show that any isomorphism between the graphs must have an action on the partition $\{B_{i,j,k}\}$: that is, for any isomorphism $\varphi$, $B_{i,j,k}^\varphi=B_{i',j',k'}$ for some $i',j',k'$ that depends on $i,j,k$.

\begin{lem}\label{lem-mut-nbrs-outside}
Suppose $(i,j,k) \neq (0,0,0)$. Then for any element $g \in B_{i,j,k}$, the number of mutual neighbours of $e$ and $g$ in either $\Gamma_1$ or $\Gamma_2$ is at most $587$.  
\end{lem}

\begin{proof} 
In both graphs, there are six sets $B_{x,y,z}$ in which $e$ has no neighbours: the sets for which $\{x,y,z\}=\{0,1,2\}
$.  Furthermore there is an additional set $B_{0,0,0}$ in which $e$ has only $10$ neighbours (the elements of $S_{0,0,0}\cup S_{0,0,0}^{-1}=T_{0,0,0}\cup T_{0,0,0}^{-1}$). Twenty sets $B_{x,y,z}$ remain.

Suppose that $B_{x,y,z}$ is one of these sets, and contains mutual neighbours of $e$ and $g$. The neighbours of $e$ in $\Gamma_1$ (of which there are at most $81$) are given by either $S_{x,y,z}$ or $S_{-x,-y,-z}$ (whichever is defined). Similarly, the neighbours of $g$ are the elements of $g+S_{x-i,y-j,z-k}$ or of $g+S_{i-x,j-y,k-z}$ (whichever is defined). It is tedious but straightforward to verify that unless $(i,j,k)=(x,y,z)$ or $(i,j,k)=-(x,y,z)$ then the definitions of these sets force the set of mutual neighbours to include at most $27$ of the neighbours of $e$. If $(i,j,k)=(x,y,z)$ then $B_{x,y,z}$ has at most 10 neighbours of $g$, but $B_{-x,-y,-z}$ may contain $81$ mutual neighbours of $e$ and $g$, so these two blocks may have up to $91$ mutual neighbours rather than the $54$ we get from any other pair of blocks that have neighbours of $e$ and $g$. The same argument with the other connection set gives the same bound in $\Gamma_2$.

This gives a bound of at most $10+10+81+18\cdot 27=587$ mutual neighbours of $e$ and $g$ in either graph.
\end{proof}

In~\Cref{lem-mut-nbrs} (which is somewhat tedious and technical) we will specify how many mutual neighbours various vertices of $B_{0,0,0}$ have with the identity vertex $e$ in $\Gamma_1$ and in $\Gamma_2$. Since any group automorphism of $G$ fixes $e$, it must map a vertex that has $k$ mutual neighbours with $e$ to a vertex that has $k$ mutual neighbours with $e$, so this lemma will be critical in limiting the possible actions of our group automorphism. Before we get there, we first prove an easy lemma showing that as long as $g \in B_{0,0,0}$, the number of mutual neighbours of $g$ and $e$ is the same  in both graphs, so we don't have to calculate these values separately.

\begin{lem}\label{mut-nbrs-same}
Suppose $g \in B_{0,0,0}$. Then the number of mutual neighbours of $e$ and $g$ in $\Gamma_1$ is the same as the number of mutual neighbours of $e$ and $g$ in $\Gamma_2$.
\end{lem}

\begin{proof}
For every $i,j,k$, we have $T_{i,j,k}=S_{i,j,k}+m_{i,j,k}v_5$ where $m_{i,j,k} \in \{0,1,2\}$. Now, $h$ is a mutual neighbour of $e$ and $g$ iff $h \in S_{i,j,k}\cap (S_{i,j,k}+g)$. This is true if and only if $h+m_{i,j,k}v_5 \in T_{i,j,k} \cap (T_{i,j,k}+g)$, which is true if and only if $h+m_{i,j,k}v_5$ is a mutual neighbour of $e$ and $g$ in $\Gamma_2$.
\end{proof}

\begin{lem}\label{lem-mut-nbrs}
The number of mutual neighbours that each vertex $v$ in the table below has with $e$ in each of $\Gamma_1$ and $\Gamma_2$, is as given in the corresponding column.

$$\begin{array}{c||c|c|c|c|c|c|c|c}
\textbf{v} &  v_1-v_5& v_2+v_3-v_4+v_5 & v_3-v_4 +v_5&v_4+v_5 &v_5 &  v_2+v_5&v_4&v_5-v_4   \\\hline
\textbf{\#} &865 & 163 & 487 & 811 &703&  702& 650 & 812 
\end{array}$$
$$\begin{array}{c||c|c|c|c|c|c}
\textbf{v} & v_1& -v_1-v_5 & v_3+v_4+v_5 & -v_3-v_5&v_2-v_3-v_4+v_5&-v_2+v_4-v_5\\\hline
\textbf{\#} & 380& 704 & 486 & 810&648&486
\end{array}$$
\end{lem}

\begin{proof}
By~\Cref{mut-nbrs-same}, the number of mutual neighbours of $e$ with any vertex of $B_{0,0,0}$ is the same in $\Gamma_1$ as in $\Gamma_2$, so we need only count one of these. We will count the mutual neighbours in $\Gamma_1$.

\begin{itemize}
\item $v_1-v_5$: Mutual neighbours with $e$ are $v_5-v_1$; $S_{1,0,0}^{\pm 1}$; $S_{0,1,0}^{\pm1}$, $S_{1,1,0}^{\pm1}$, $S_{0,1,1}^{\pm1}$, $S_{1,1,1}^{\pm1}$, $S_{2,1,1}^{\pm1}$. The total number of these is $$1+54+81\cdot 10=865.$$ 

\item $v_2+v_3-v_4+v_5$: Mutual neighbours with $e$ are $-v_2-v_3+v_4-v_5$; $S_{0,0,1}^{\pm1}$. The total number of these is $$1+81\cdot 6=163.$$ 

\item $v_3-v_4+v_5$: Mutual neighbours with $e$ are $-v_3+v_4-v_5$; $S_{0,1,0}^{\pm1}$, $S_{0,0,1}^{\pm1}$, $S_{1,2,1}^{\pm1}$. The total number of these is $$1+81\cdot6=487.$$ vertices.

\item $v_4+v_5$: Mutual neighbours with $e$ are $-v_4-v_5$; $S_{0,1,0}^{\pm1}$, $S_{0,0,1}^{\pm1}$, $S_{1,0,1}^{\pm1}$, $S_{1,1,1}^{\pm1}$, $S_{1,2,1}^{\pm1}$. The total number of these is $$1+81\cdot 10=811.$$

\item $v_5$: Mutual neighbours with $e$ are $-v_5$; $S_{1,0,0}^{\pm1}$; $S_{0,1,0}^{\pm1}$, $S_{0,0,1}^{\pm1}$, $S_{1,1,0}^{\pm1}$, $S_{1,0,1}^{\pm1}$. The total number of these is $$1+54+81\cdot 8=703.$$

\item $v_2+v_5$: Mutual neighbours with $e$ are $S_{1,0,0}^{\pm1}$; $S_{0,0,1}^{\pm1}$, $S_{1,0,1}^{\pm1}$, $S_{1,2,1}^{\pm1}$, $S_{1,1,2}^{\pm1}$. The total number of these is $$54+81\cdot 8=702.$$   

\item $v_4$: Mutual neighbours with $e$ are $v_4+v_5$; $-v_5$; $S_{0,1,0}^{\pm1}$, $S_{0,0,1}^{\pm1}$, $S_{1,0,1}^{\pm1}$, $S_{0,1,1}^{\pm1}$. The total number of these is $$2+81\cdot 8 = 650.$$ 

\item $v_5-v_4$: Mutual neighbours with $e$ are $-v_5$; $-v_4-v_5$; $S_{0,1,0}^{\pm1}$, $S_{0,0,1}^{\pm1}$, $S_{1,0,1}^{\pm1}$, $S_{2,1,1}^{\pm1}$, $S_{1,1,2^{\pm1}}$. The total number of these is $$2+81\cdot 10=812.$$ 

\item $v_1$: Mutual neighbours with $e$ are $v_5$; $v_1-v_5$; $S_{1,0,0}^{\pm1}$; $S_{0,1,0}^{\pm1}$, $S_{1,1,0}^{\pm1}$. The total number of these is $$2+54+81\cdot 4=380.$$

\item $-v_1-v_5$: Mutual neighbours with $e$ are $v_5$; $v_5-v_1$; $S_{1,0,0}^{\pm1}$; $S_{0,1,0}^{\pm1}$, $S_{1,1,0}^{\pm1}$, $S_{1,2,1}^{\pm1}$, $S_{1,1,2}^{\pm1}$. The total number of these is $$2+54+81\cdot 8=704.$$ 

\item $v_3+v_4+v_5$: Mutual neighbours with $e$ are $S_{0,1,0}^{\pm1}$, $S_{0,0,1}^{\pm1}$, $S_{2,1,1}^{\pm1}$. The total number of these is $81\cdot 6=486$ vertices.

\item $-v_3-v_5$: Mutual neighbours with $e$ are $S_{0,1,0}^{\pm1}$, $S_{0,0,1}^{\pm1}$, $S_{1,1,0}^{\pm1}$, $S_{1,1,1}^{\pm1}$, $S_{1,1,2}^{\pm1}$. The total number of these is $$81\cdot 10=810.$$ 

\item $v_2-v_3-v_4+v_5$: Mutual neighbours with $e$ are $S_{0,0,1}^{\pm1}$, $S_{2,1,1}^{\pm1}$, $S_{1,2,1}^{\pm1}$, $S_{1,1,2}^{\pm1}$. The total number of these is $$81\cdot 8=648.$$ 

\item $-v_2+v_4-v_5$: Mutual neighbours with $e$ are $S_{0,0,1}^{\pm1}$, $S_{1,0,1}^{\pm1}$, $S_{1,1,1}^{\pm1}$. The total number of these is $$81\cdot 6=486.$$ 
\end{itemize}
\end{proof}

With this result in hand, we can show that a group automorphism acting as a graph isomorphism must either fix some of the elements of $S_{0,0,0}^{\pm1}$ pointwise, and have very small potential orbits on the other elements, or must have these properties after being combined with another group automorphism that acts as a graph automorphism of $\Gamma_1$.

\begin{lem}\label{aut-fixes-5}
If there is a group automorphism $\alpha$ of $G$ that maps $S\cup S^{-1}$ to $T\cup T^{-1}$, then there is one that fixes $v_5$ and $v_5^{-1}$. Moreover, any such automorphism fixes each of the sets $\{v_1-v_5,v_5-v_1\}$, $\{v_2+v_3-v_4+v_5,-v_2-v_3+v_4-v_5\}$, $\{v_3-v_4+v_5,-v_3+v_4-v_5\}$, and $\{v_4+v_5, -v_4-v_5\}$.
\end{lem}

\begin{proof}
Any group automorphism fixes $e$. A group automorphism $\alpha$ that maps $S\cup S^{-1}$ to $T\cup T^{-1}$ must map any vertex that has a certain number of mutual neighbours with $e$ in $\Gamma_1$, to a vertex that has that number of mutual neighbours with $e$ in $\Gamma_2$.

By~\Cref{lem-mut-nbrs}, the vertices $v_1-v_5$, $v_2+v_5$, $-v_3-v_5$, $v_4+v_5$, and $v_5$ each has more than $587$ mutual neighbours with $e$, so by~\Cref{lem-mut-nbrs-outside} each of these vertices must map to a vertex in $B_{0,0,0}$, and must be the image of a vertex in $B_{0,0,0}$. Note that $B_{0,0,0}$ is actually a subgroup of $G$, and is generated by these five elements. This implies that $\alpha$ must map this generating set for $B_{0,0,0}$ to a generating set for $B_{0,0,0}$, so $B_{0,0,0}^\alpha=B_{0,0,0}$.

For convenience in this paragraph, let $S_0=S_{0,0,0}\cup S_{0,0,0}^{-1}$ and $T_0=T_{0,0,0}\cup T_{0,0,0}^{-1}$. Since $$(S\cup S^{-1})\cap B_{0,0,0}=S_0=T_0=(T\cup T^{-1}) \cap B_{0,0,0}$$ and $B_{0,0,0}^\alpha=B_{0,0,0}$,
we must have $S_0^\alpha=T_0=S_0$.
In particular, since the numbers of mutual neighbours that each vertex in each inverse-closed pair in $S_0$ has with $e$ is distinct from the number of mutual neighbours that each vertex of any other inverse-closed pair in $S_0$ has with $e$ (see~\Cref{lem-mut-nbrs}), each inverse-closed pair in $S_0$ must be mapped to the same inverse-closed pair in $S_0$. This completes the proof unless $v_5^\alpha=v_5^{-1}$ (in which case $(v_5^{-1})^\alpha=v_5$), which we now assume.

Let $\sigma$ be the automorphism of $G$ that inverts every element of $G$; note that $\sigma$ fixes $S\cup S^{-1}$, and therefore $\sigma\alpha$ maps $S\cup S^{-1}$ to $T\cup T^{-1}$. Now $v_5^{\sigma\alpha}=v_5$, and $\sigma\alpha$ also fixes each of the other inverse-closed pairs separately since $\alpha$ did, completing the proof.
\end{proof}

In fact, we can now show that a group automorphism acting as a graph isomorphism must fix every element of $S_{0,0,0}^{\pm1}$ pointwise (possibly after being combined in the preceding lemma with another group automorphism that acts as a graph automorphism).

\begin{lem}\label{aut-fixes-B_0}
If there is a group automorphism $\alpha$ of $G$ that maps $S\cup S^{-1}$ to $T\cup T^{-1}$ and fixes $v_5$ and $v_5^{-1}$, then it must also fix every other vertex of $S_{0,0,0}\cup S_{0,0,0}^{-1}$.
\end{lem}

\begin{proof}
We first show that it fixes $v_4+v_5$ and its inverse.

Any such $\alpha$ induces a graph automorphism from $\Gamma_1$ to $\Gamma_2$. Since $\Gamma_1$ is a Cayley graph, the number of mutual neighbours of $v_5$ with $v_4+v_5$ is the same as the number of mutual neighbours of $e$ with $v_4$. Likewise, the number of mutual neighbours of $v_5$ with $-v_4-v_5$ is the same as the number of mutual neighbours of $e$ with $v_5-v_4$. Since we see from~\Cref{lem-mut-nbrs} that these numbers are not equal, $\alpha$ must fix both $v_4+v_5$ and $-v_4-v_5$.

Next we show that $v_1-v_5$ and its inverse are fixed.

Since $\Gamma_1$ is a Cayley graph, the number of mutual neighbours of $-v_5$ with $v_1-v_5$ is the same as the number of mutual neighbours of $e$ with $v_1$. Likewise, the number of mutual neighbours of $-v_5$ with $v_5-v_1$ is the same as the number of mutual neighbours of $e$ with $-v_1-v_5$. Since we see from~\Cref{lem-mut-nbrs} that these numbers are not equal, $\alpha$ must fix both $v_1-v_5$ and $v_5-v_1$.

Now we show that $v_3-v_4+v_5$ and its inverse are fixed. At this point, notice that since $\alpha$ fixes $v_5$ and $v_4+v_5$, it also fixes $v_4$. 

Since $\Gamma_1$ is a Cayley graph, the number of mutual neighbours of $v_4$ with $v_3-v_4+v_5$ is the same as the number of mutual neighbours of $e$ with $v_3+v_4+v_5$. Likewise, the number of mutual neighbours of $v_4$ with $-v_3+v_4-v_5$ is the same as the number of mutual neighbours of $e$ with $-v_3-v_5$. Since we see from~\Cref{lem-mut-nbrs} that these numbers are not equal, $\alpha$ must fix both $v_3-v_4+v_5$ and $-v_3+v_4-v_5$.

Finally, we must show that $v_2+v_3-v_4+v_5$ and its inverse are fixed. At this point, notice that since $\alpha$ fixes $v_4$, $v_5$, and $v_3-v_4+v_5$, it also fixes $-v_3$.

Since $\Gamma_1$ is a Cayley graph, the number of mutual neighbours of $-v_3$ with $v_2+v_3-v_4+v_5$ is the same as the number of mutual neighbours of $e$ with $v_2-v_3-v_4+v_5$. Likewise, the number of mutual neighbours of $-v_3$ with $-v_2-v_3+v_4-v_5$ is the same as the number of mutual neighbours of $e$ with $-v_2+v_4-v_5$. Since we see from~\Cref{lem-mut-nbrs} that these numbers are not equal, $\alpha$ must fix both $v_2+v_3-v_4+v_5$ and $-v_2-v_3+v_4-v_5$.
\end{proof}

At this point we have enough information to show the invariant action on the partition $\{B_{i,j,k}\}$, as mentioned earlier.

\begin{lem}\label{V-blocks}
Any group automorphism $\alpha$ of $G$ that fixes every vertex of $B_{0,0,0}$ has an invariant action on the collection $\{B_{i,j,k}:0 \le i,j,k \le 2\}$.
\end{lem}

\begin{proof}
Since $\alpha$ is an automorphism of $G$, once we know its action on a generating set for $G$, this completely determines its action. Since $\alpha$ fixes every vertex of $B_{0,0,0}$, it fixes each $v_i$ for $1\le i \le 5$. This means that if $g \in B_{i,j,k}$ with $g=iw_1+jw_2+kw+3+h$ then $g^\alpha=iw_1^\alpha+jw_2^\alpha+kw_3^\alpha+h$ since $h \in B_{0,0,0}$. So $B_{i,j,k}^\alpha$ has the form $B_{i',j',k'}$ for some $0 \le i',j',k' \le 2$.
\end{proof}

In fact, the invariant action on the partition can have only two possible forms.

\begin{lem}\label{W-action}
If there is a group automorphism $\alpha$ of $G$ that maps $S\cup S^{-1}$ to $T\cup T^{-1}$ and fixes every vertex of  $B_{0,0,0}$, then it must either leave every $B_{i,j,k}$ invariant, or map every $B_{i,j,k}$ to $B_{-i,-j,-k}$.
\end{lem}

\begin{proof}
By~\Cref{V-blocks}, $\alpha$ must have an invariant action on the collection $\{B_{i,j,k}:0 \le i,j,k \le 2\}$.

The sets $B_{1,0,0}$ and $B_{-1,0,0}$ are the only sets that have exactly $27$ neighbours of $e$ in both $\Gamma_1$ and $\Gamma_2$, so must be left invariant by $\alpha$.

The sets $B_{0,1,0}$ and $B_{0,-1,0}$ are the only sets for which all neighbours of $e$ in both $\Gamma_1$ and $\Gamma_2$ are also neighbours of $iv_1+jv_3+kv_4+\ell v_5$ for every $0 \le i,j,k,\ell \le 2$, so these sets must be left invariant by $\alpha$.

Similarly, the sets $B_{0,0,1}$ and $B_{0,0,-1}$ are the only sets for which all neighbours of $e$ in both $\Gamma_1$ and $\Gamma_2$ are also neighbours of $iv_2+jv_3+kv_4+\ell v_5$ for every $0 \le i,j,k,\ell \le 2$, so these sets must be left invariant by $\alpha$.

Likewise, the sets $B_{1,1,1}$ and $B_{-1,-1,-1}$ are the only sets for which (in both $\Gamma_1$ and $\Gamma_2$) all neighbours of $e$ are also neighbours of $iv_1+jv_2+kv_3+\ell v_4$ for every $0 \le i,j,k,\ell \le 2$. Notice this is true even though $S_{1,1,1}\neq T_{1,1,1}$. So these sets must be left invariant by $\alpha$.

Taken together, these force $w_1^\alpha \in B_{\pm1,0,0}$; $w_2^\alpha  \in B_{0,\pm1,0}$; and $w_3^\alpha  \in B_{0,0,\pm1}$. However, the previous paragraph also tells us that $(w_1+w_2+w_3)^\alpha \in B_{1,1,1}\cup B_{-1,-1,-1}$. This forces the signs of each of the $\pm$ in our first three conclusions to be the same, yielding the desired conclusion.
\end{proof}

We can even show that after possibly combining with a group automorphism that acts as a graph automorphism, the invariant action fixes every set in the partition.

\begin{lem}\label{aut-Bs-invariant}
If there is a group automorphism $\alpha$ of $G$ that maps $S\cup S^{-1}$ to $T\cup T^{-1}$ and fixes every vertex of  $B_{0,0,0}$, then there is one that leaves every $B_{i,j,k}$ invariant.
\end{lem}

\begin{proof}
By~\Cref{W-action}, the only other possibility is that $\alpha$ maps every $B_{i,j,k}$ to $B_{-i,-j,-k}$.

Let $\sigma_W$ be the automorphism of $G$ determined by $w_i^{\sigma_W}=w_i^{-1}$ for every $1 \le i \le 3$, and $v_i^{\sigma_W}=v_i$ for every $1 \le i \le 5$. Note that $\sigma_W$ fixes $S$, and therefore $\sigma_W\alpha$ maps $S$ to $T$. Also $\sigma_W$ fixes every vertex of $B_{0,0,0}$ and maps every $B_{i,j,k}$ to $B_{-i,-j,-k}$. Therefore $\sigma_W\alpha$ has all of the properties we claimed.
\end{proof}

In our final lemma, we show that a group automorphism that has all of the properties we have been deducing, cannot exist.

\begin{lem}\label{lem-no-aut}
A group automorphism $\alpha$ of $G$ that maps $S\cup S^{-1}$ to $T\cup T^{-1}$, fixes every vertex of  $B_{0,0,0}$, and  leaves every $B_{i,j,k}$ invariant must map $(S\setminus S_{0,0,0})\cup S'_{0,0,0} $ to $(T \setminus T_{0,0,0})\cup S'_{0,0,0}$ for any $S'_{0,0,0} \subseteq B_{0,0,0}$. In fact, there is no such group automorphism.
\end{lem}

\begin{proof}
Notice that for every $(i,j,k)\neq (0,0,0)$ we have either $B_{i,j,k} \cap S_{i,j,k}=\emptyset$ or $B_{i,j,k} \cap S_{i,j,k}^{-1} =\emptyset$ (or both), and $T_{i,j,k} =\emptyset$ if and only if $S_{i,j,k}=\emptyset$. Since $B_{i,j,k}^\alpha=B_{i,j,k}$, this means that $S_{i,j,k}^\alpha=T_{i,j,k}$. Hence, since $\alpha$ fixes $B_{0,0,0}$ pointwise and $S_{0,0,0}=T_{0,0,0}$, we have $S^\alpha=T$. In fact, since $\alpha$ fixes $B_{0,0,0}$ pointwise, for any $S'_{0,0,0} \subseteq B_{0,0,0}$ we have $$((S \setminus T_{0,0,0})\cup S'_{0,0,0} )^\alpha=(T\setminus T_{0,0,0})\cup S'_{0,0,0} ,$$ as claimed.

Since the connection sets of $\overrightarrow{\Gamma_1'}$ and $\overrightarrow{\Gamma_2'}$ have this form, such an $\alpha$ would act as a digraph isomorphism between them, but since \cite{Spiga-CI-3} proved these digraphs are not isomorphic via any group automorphism, this is impossible. Hence no such $\alpha$ exists.
\end{proof}

We pull all of our results in this section together in our conclusion.

\begin{cor}\label{no-aut}
There is no automorphism of $G$ that maps $S \cup S^{-1}$ to $T \cup T^{-1}$.
\end{cor}

\begin{proof}
By~\Cref{aut-fixes-5}, if there were such an automorphism, then there must be one that fixes $v_5$ and $v_5^{-1}$. Furthermore, by~\Cref{aut-fixes-B_0}, any such automorphism fixes every vertex of $S_{0,0,0}\cup S_{0,0,0}^{-1}$. Since it is a group automorphism and fixes every element of $S_{0,0,0}$, and this generates $B_{0,0,0}$, it must fix every vertex of $B_{0,0,0}$.

By~\Cref{aut-Bs-invariant}, this would imply the existence of an automorphism that maps $S\cup S^{-1}$ to $T\cup T^{-1}$, fixes every vertex of $B_{0,0,0}$, and leaves every $B_{i,j,k}$ invariant. Finally, by~\Cref{lem-no-aut}, no such automorphism exists.
\end{proof}

Putting this result together with~\Cref{prop-iso}, and using the well-known fact that every subgroup of a CI-group must be a CI-group, yields the main result of this paper. 

\begin{thm}
The group $\mathbb Z_3^8$ is not a CI-group. Neither is any group containing $\mathbb Z_3^8$ as a subgroup. In particular, $\mathbb Z_3^n$ is not a CI-group for $n \ge 8$.
\end{thm}

\end{document}